\documentclass[11pt,reqno]{article}
\usepackage{amsmath,amssymb,amsthm}
\usepackage{color,url}
\usepackage{graphicx}
\newtheorem{theorem}{Theorem}[section]

\newtheorem{corollary}[theorem]{Corollary}

\newtheorem{lemma}[theorem]{Lemma}
\newtheorem{example}[theorem]{Example}
\newtheorem{remark}[theorem]{Remark}

\numberwithin{equation}{section}

\usepackage
{hyperref}

\def\bR{\mathbb{R}}
\def\bC{\mathbb{C}}
\def\bN{\mathbb{N}}
\def\bH{\mathbb{H}}
\def\bP{\mathbb{P}}
\def\bM{\mathbb{M}}
\def\Tr{\mathrm{Tr}}
\def\eps{\varepsilon}

\begin{document}
\allowdisplaybreaks

\centerline{\LARGE Log-majorization and matrix norm inequalities}
\medskip
\centerline{\LARGE with application to quantum information}

\bigskip
\bigskip
\centerline{\large
Fumio Hiai\footnote{{\it E-mail:} hiai.fumio@gmail.com
\ \ {\it ORCID:} 0000-0002-0026-8942}}

\medskip
\begin{center}
$^1$\,Graduate School of Information Sciences, Tohoku University, \\
Aoba-ku, Sendai 980-8579, Japan
\end{center}

\medskip

\begin{abstract}
We are concerned with log-majorization for matrices in connection with the multivariate Golden--Thompson
trace inequality and the Karcher mean (i.e., a multivariate extension of the weighted geometric mean).
We show an extension of Araki's log-majorization and apply it to the $\alpha$-$z$-R\'enyi divergence in
quantum information. We discuss the equality cases in the multivariate trace inequality of Golden--Thompson
type and in the norm inequality for the Karcher mean. The paper includes an appendix to correct the proof
of the author's old result on the equality case in the norm inequality for the weighted geometric mean.



\end{abstract}

\section{Introduction}\label{Sec-1}

Several types of majorizations originally due to Hardy, Littlewood and P\'olya are known for the
eigenvalues and the singular values of matrices (also Hilbert space operators), as was fully clarified in the
book \cite{MOA} (also in \cite{An,Bh,Hi0,Hi1} and many others). These majorizations give rise to powerful
tools in deriving various matrix norm and trace inequalities for matrices. The notion of majorization of
multiplicative type called log-majorization (whose definition is given in Sec.~\ref{Sec-2}), as well as
the additive type (weak) majorization, is sometimes very useful in matrix theory. For instance, the
log-majoriation of Golden--Thompson type given by Araki \cite{Ar} has been used to generalize the classical
Golden--Thompson and the Araki--Lieb--Thirring trace inequalities with applications to matrix analysis,
mathematical physics, quantum information, etc. On the other hand, the log-majorization shown in \cite{AH}
for the weighted geometric mean is regarded as complementary to the above type, giving norm inequalities of
complementary Golden--Thompson type. In \cite{Hi} we characterized the equality cases in the norm
inequalities derived from the above-mentioned Araki's log-majorization and its complementary version.

In the present paper we obtain miscellaneous results related to log-majorization. First, in Sec.~\ref{Sec-2}
(Theorem \ref{T-2.1}) we show an extension of Araki's log-majorization which is similar to that given in \cite{Hi2}.
The $\alpha$-$z$-R\'enyi divergence $D_{\alpha,z}(\rho\|\sigma)$ is defined for positive semidefinite matrices
$\rho,\sigma$ with two parameters $\alpha\ge0$ and $z>0$, that has recently been in active consideration in
quantum information. The monotonicity of $D_{\alpha,z}(\rho\|\sigma)$ in $z$ is easy from Araki's
log-majorization, as discussed in \cite{LT,MH2}. In Sec.~\ref{Sec-3} (Theorem \ref{T-3.1}), by use of the
extended Araki's log-majorization in Sec.~\ref{Sec-2} we prove that $D_{\alpha,z}(\rho\|\sigma)$ is monotone
increasing in the parameter $\alpha$ for any fixed $z$.

The classical Golden--Thompson trace inequality $\Tr\,e^{H_1+H_2}\le\Tr\,e^{H_1}e^{H_2}$ for Hermitian
matrices $H_1,H_2$ cannot be extended in this form to more than two matrices, while an integral form for
three matrices was found by Lieb in \cite{Li}. In \cite{SBT} the authors established the trace inequality in the
form of an integral average, extending the Golden--Thompson and the Araki--Lieb--Thirring inequalities
as well as Lieb's triple matrix version to arbitrary many matrices, and applied it to recoverability questions in
quantum information. In \cite{HKT} we generalized the trace inequality in \cite{SBT} to the form of
log-majorization. Then it might be interesting to find whether or not the characterization of the equality case in
the norm inequality of Golden--Thompson type in \cite{Hi} can be extended to the multivariable case. In
Sec.~\ref{Sec-4} (Example \ref{E-4.1}) we provide an example suggesting that a similar characterization of
the equality case is not possible for the multivariate trace inequality.

The (weighted) geometric mean, first introduced in \cite{PW}, is the most studied operator mean in
Kubo--Ando's sense \cite{KA}. Its multivariate extension was developed in the Riemannian geometry approach
in \cite{BH,LL1,Mo}, which has been further studied by many authors in the name of the Karcher mean.
Moreover, the log-majorization for the Karcher mean was observed in \cite{HP} based on the works \cite{BK,Ya},
extending that for the weighted geometric mean in \cite{AH}. In Sec.~\ref{Sec-5} (Theorem \ref{T-5.4}) we
characterize the equality case in the norm inequality of the Karcher mean derived from log-majorization. The
characterizations are similar to those in the two-variable case in \cite[Theorem 3.1]{Hi}, though not fully
described as those in \cite{Hi} due to the fact that analyticity of the Karcher mean is not known. In
Sec.~\ref{Sec-6} we examine the analyticity question of the Karcher mean.

This paper includes an appendix, where we update the proof of \cite[Theorem 2.1]{Hi}. As a matter of fact, we
have recently found a serious gap in the proof of \cite{Hi}, so that we take this opportunity to fix it.

\section{Extension of Araki's log-majorization}\label{Sec-2}

We write $\bM_m$ for the $m\times m$ matrices and $\bM_m^+$ for the $m\times m$ positive semidefinite
matrices. We write $A>0$ if $A\in\bM_m$ is positive definite. For each $A\in\bM_m^+$ let
$\lambda(A)=(\lambda_1(A),\dots,\lambda_m(A))$ denote the eigenvalues of $A$ in decreasing order (i.e.,
$\lambda_1(A)\ge\dots\ge\lambda_m(A)$) counting multiplicities. Also, for each $X\in\bM_m$ let
$s(X)=(s_1(X),\dots,s_m(X))$ denote the singular values of $X$ in decreasing order counting multiplicities.
The usual trace of $X\in\bM_m$ is denoted by $\Tr\,X$.

For $a=(a_1,\dots,a_m)$ and $b=(b_1,\dots,b_m)$ in $\bR_+^m=[0,\infty)^m$ the \emph{log-majorization}
$a\prec_{\log}b$ is defined as
\[
\prod_{i=1}^ka_i^\downarrow\le\prod_{i=1}^kb_i^\downarrow\quad\mbox{for}\ k=1,\dots,m-1
\]
and
\[
\prod_{i=1}^ma_i^\downarrow=\prod_{i=1}^mb_i^\downarrow,\quad
\mbox{i.e.,}\quad\prod_{i=1}^ma_i=\prod_{i=1}^mb_i,
\]
where $(a_1^\downarrow,\dots,a_m^\downarrow)$ is the decreasing rearrangement of $(a_1,\dots,a_m)$. For
$A,B\in\bM_m^+$ we write the log-majorization $A\prec_{\log}B$ if $\lambda(A)\prec_{\log}\lambda(B)$, that is,
$\prod_{i=1}^k\lambda_i(A)\le\prod_{i=1}^k\lambda_i(B)$ for $k=1,\dots,m-1$ and $\det A=\det B$. For details on
majorization theory for matrices; see, e.g., \cite{An,Bh,Hi1,MOA}. In particular, it is well known
\cite[Proposition 4.4.13]{Hi1} that $A\prec_{\log}B$ implies the \emph{weak majorization} $A\prec_wB$, i.e.,
$\sum_{i=1}^k\lambda_i(A)\le\sum_{i=1}^k\lambda_i(B)$ for $k=1,\dots,m$, and the latter is equivalent to that
$\|A\|\le\|B\|$ holds for all unitarily invariant norms $\|\cdot\|$, where a norm $\|\cdot\|$ on $\bM_m$ is said to be
\emph{unitarily invariant} if $\|UXV\|=\|X\|$, $X\in\bM_m$, for all unitaries $U,V$.

Araki's log-majorization \cite{Ar} (also \cite[Theorem A]{AH}) says that for every $A,B\in\bM_m^+$,
\begin{align}\label{F-2.1}
A^{p/2}B^pA^{p/2}\prec_{\log}(A^{1/2}BA^{1/2})^p,\qquad0<p\le1,
\end{align}
or equivalently,
\[
(A^{p/2}B^pA^{p/2})^{1/p}\prec_{\log}(A^{q/2}B^qA^{q/2})^{1/q},\qquad0<p\le q,
\]
which gives
\begin{align}\label{F-2.2}
\big\|(A^{p/2}B^pA^{p/2})^{1/p}\big\|\le\big\|(A^{q/2}B^qA^{q/2})^{1/q}\big\|,\qquad0<p\le q,
\end{align}
for any unitarily invariant norm $\|\cdot\|$. In particular, when $\|\cdot\|$ is the trace-norm, this becomes the
so-called Araki--Lieb--Thirring trace inequality \cite{Ar}. In \cite[Theorem 2.1]{Hi} we characterized the equality
case in inequality \eqref{F-2.2} when $\|\cdot\|$ is strictly increasing, i.e., for any $A,B\in\bM_m^+$ with
$A\le B$, $\|A\|=\|B\|$ implies $A=B$. For instance, the Schatten $p$-norms
$\|X\|_p:=(\Tr\,|X|^p)^{1/p}$ for $1\le p<\infty$ (in particular, the trace-norm) are
strictly increasing. In Appendix~\ref{Sec-A} of this paper we supply the corrected version of the proof of
\cite[Theorem 2.1]{Hi}, because the original proof contains a serious gap.\footnote{
Although \cite{Hi} is an old paper from 30 years ago, the author realized the gap quite recently in the course
of preparing the present paper, and came up to his idea to update the proof on this occasion.}

A generalization of Araki's log-majorization was given in \cite{Hi2}. The following is another generalization in a
similar vein.

\begin{theorem}\label{T-2.1}
Let $A_j,B_j\in\bM_m^+$, $j=1,2$, be such that $A_1A_2=A_2A_1$ and $B_1B_2=B_2B_1$. Then for any
$\theta\in[0,1]$,
\begin{align}\label{F-2.3}
\lambda\bigl(\bigl(A_1^\theta A_2^{1-\theta}\bigr)^{1/2}\bigl(B_1^\theta B_2^{1-\theta}\bigr)
\bigl(A_1^\theta A_2^{1-\theta}\bigr)^{1/2}\bigr)
\prec_{\log}\lambda^\theta\bigl(A_1^{1/2}B_1A_1^{1/2}\bigr)
\lambda^{1-\theta}\bigl(A_2^{1/2}B_2A_2^{1/2}\bigr),
\end{align}
where
\[
\lambda^\theta\bigl(A_1^{1/2}B_1A_1^{1/2}\bigr)\lambda^{1-\theta}\bigl(A_2^{1/2}B_2A_2^{1/2}\bigr)
:=\bigl(\lambda_i\bigl(A_1^{1/2}B_1A_1^{1/2}\bigr)^\theta
\lambda_i\bigl(A_2^{1/2}B_2A_2^{1/2}\bigr)^{1-\theta}\bigr)_{i=1}^m.
\]
For $\theta=0,1$ in the above, we can adopt either convention of [$A_j^0:=I$, $0^0:=1$ (for scalar)] or
[$A_j^0:=\mbox{the support projection of $A_j$}$, $0^0=0$ (for scalar)].
\end{theorem}

\begin{proof}
Assume that $0<\theta<1$. By taking the limit from $A_j+\eps I,B_j+\eps I$ we may assume that $A_j,B_j>0$. First,
let us prove the inequality
\begin{align}\label{F-2.4}
\big\|\bigl(A_1^\theta A_2^{1-\theta}\bigr)^{1/2}\bigl(B_1^\theta B_2^{1-\theta}\bigr)
\bigl(A_1^\theta A_2^{1-\theta}\bigr)^{1/2}\big\|_\infty
\le\big\|A_1^{1/2}B_1A_1^{1/2}\big\|_\infty^\theta\big\|A_2^{1/2}B_2A_2^{1/2}\big\|_\infty^{1-\theta},
\end{align}
where $\|\cdot\|_\infty$ stands for the operator norm. To do this, by replacing $B_j$ with $b_jB_j$ for some $b_j>0$,
$j=1,2$, it suffices to assume that $\big\|A_j^{1/2}B_jA_j^{1/2}\big\|_\infty=1$ so that $B_j\le A_j^{-1}$, $j=1,2$.
Thanks to the commutativity assumption and joint monotonicity of operator means in Kubo--Ando's sense
\cite{KA,Hi1}, we then have
\[
B_1^\theta B_2^{1-\theta}=B_1\#_{1-\theta} B_2\le A_1^{-1}\#_{1-\theta} A_2^{-1}
=\bigl(A_1^\theta A_2^{1-\theta}\bigr)^{-1},
\]
where $\#_{1-\theta}$ stands for the $(1-\theta)$-weighted geometric mean. Hence
\[
\big\|\bigl(A_1^\theta A_2^{1-\theta}\bigr)^{1/2}\bigl(B_1^\theta B_2^{1-\theta}\bigr)
\bigl(A_1^\theta A_2^{1-\theta}\bigr)^{1/2}\big\|_\infty\le1
\]
so that \eqref{F-2.4} follows.

Next, for any $k=1,\dots,n$, consider the $k$-fold antisymmetric tensor powers $A_j^{\wedge k},B_j^{\wedge k}$.
Apply \eqref{F-2.4} to $A_j^{\wedge k},B_j^{\wedge k}$ and use \cite[Lemmas 4.6.2, 4.6.3]{Hi1} to obtain
\[
\prod_{i=1}^k\lambda_i\bigl(\bigl(A_1^\theta A_2^{1-\theta}\bigr)^{1/2}\bigl(B_1^\theta B_2^{1-\theta}\bigr)
\bigl(A_1^\theta A_2^{1-\theta}\bigr)^{1/2}\bigr)
\le\prod_{i=1}^k\lambda_i\bigl(A_1^{1/2}B_1A_1^{1/2}\bigr)^\theta
\lambda_i\bigl(A_2^{1/2}B_2A_2^{1/2}\bigr)^{1-\theta}.
\]
Furthermore,
\begin{align*}
\det\bigl(\bigl(A_1^\theta A_2^{1-\theta}\bigr)^{1/2}\bigl(B_1^\theta B_2^{1-\theta}\bigr)
\bigl(A_1^\theta A_2^{1-\theta}\bigr)^{1/2}\bigr)
&={\det}^\theta A_1\,{\det}^{1-\theta} A_2\,{\det}^\theta B_1\,{\det}^{1-\theta} B_2 \\
&={\det}^\theta\bigl(A_1^{1/2}B_1A_2^{1/2}\bigr)\,{\det}^{1-\theta}\bigl(A_2^{1/2}B_2A_2^{1/2}\bigr).
\end{align*}
Therefore, \eqref{F-2.3} holds when $0<\theta<1$.

Finally, when $\theta=0$ or $1$, \eqref{F-2.3} holds trivially in the first convention of $A_j^0$ and $0^0$. In the
second convention, \eqref{F-2.3} immediately follows by taking the limit as $\theta\searrow0$ or $\theta\nearrow1$
from the case when $0<\theta<1$.
\end{proof}

Since $\lambda\bigl(A_j^{1/2}B_jA_j^{1/2}\bigr)=s^2\bigl(A_j^{1/2}B_j^{1/2}\bigr)$, note that log-majorization
\eqref{F-2.3} is equivalently written as
\begin{align}\label{F-2.5}
s\bigl(\bigl(A_1^\theta A_2^{1-\theta}\bigr)\bigl(B_1^\theta B_2^{1-\theta}\bigr)\bigr)
\prec_{\log}s^\theta(A_1B_1)s^{1-\theta}(A_2B_2),
\qquad0\le\theta\le1,
\end{align}
with either convention of $A_j^0$ and $0^0$ as stated in Theorem \ref{T-2.1}.

\begin{remark}\label{R-2.2}\rm
When $A_1=B_1=I$, log-majorization \eqref{F-2.3} reduces to
\[
A_2^{\theta/2}B_2^{1-\theta} A_2^{\theta/2}\prec_{\log}\bigl(A_2^{1/2}B_2A_2^{1/2}\bigr)^{1-\theta},
\qquad0\le\theta\le1,
\]
that is Araki's log-majorization in \eqref{F-2.1}. Also, when $A_2=B_1=I$, \eqref{F-2.5} reduces to
\[
s(A_1^\theta B_2^{1-\theta})\prec_{\log}s(A_1^\theta)s(B_2^{1-\theta}),
\]
that is Horn's log-majorization \cite[p.\ 338]{MOA} (also \cite[Corollary 4.3.5]{Hi1}).
\end{remark}

The next corollary is shown from \eqref{F-2.5} similarly to \cite[Corollary 3.4]{Hi2}, while we give the proof
for completeness.

\begin{corollary}\label{C-2.3}
Let $A_j,B_j$ be as in Theorem \ref{T-2.1}, and let $\|\cdot\|$ be a unitarily invariant norm on $\bM_m$. Then
for any $r>0$ and $\theta\in[0,1]$,
\begin{align}\label{F-2.6}
\big\|\,\big|\bigl(A_1^\theta A_2^{1-\theta}\bigr)\bigl(B_1^\theta B_2^{1-\theta}\bigr)\big|^r\big\|
\le\big\|\,|A_1B_1|^r\big\|^\theta\big\|\,|A_2B_2|^r\big\|^{1-\theta},
\end{align}
where we can adopt either convention of $A_j^0$ and $0^0$ as stated in Theorem \ref{T-2.1}.
\end{corollary}

\begin{proof}
Assume that $0<\theta<1$. Let $\Phi$ be the symmetric gauge function on $\bR^m$ corresponding to the
unitarily invariant norm $\|\cdot\|$ so that $\|X\|=\Phi(s(X))$ for $X\in\bM_m$; see \cite[Theorem IV.2.1]{Bh}.
From \eqref{F-2.5} and \cite[Lemma 4.4.2]{Hi1} we have
\begin{align*}
\big\|\,\big|\bigl(A_1^\theta A_2^{1-\theta}\bigr)\bigl(B_1^\theta B_2^{1-\theta}\bigr)\big|^r\big\|
&=\Phi\bigl(s^r\bigl(\bigl(A_1^\theta A_2^{1-\theta}\bigr)\bigl(B_1^\theta B_2^{1-\theta}\bigr)\bigr)\bigr) \\
&\le\Phi\bigl(s^{\theta r}(A_1B_1)s^{(1-\theta)r}(A_2B_2)\bigr) \\
&\le\Phi^\theta\bigl(s^r(A_1B_1)\bigr)\Phi^{1-\theta}\bigl(s^r(A_2B_2)\bigr) \\
&=\big\|\,|A_1B_1|^r\big\|^\theta\big\|\,|A_2B_2|^r\big\|^{1-\theta},
\end{align*}
where the second inequality above is due to \cite[Theorem IV.1.6]{Bh}.

The case $\theta=0$ or $1$ is treated as in the last part of the proof of Theorem \ref{T-2.1}.
\end{proof}

\begin{remark}\label{R-2.4}\rm
For $A,B\in\bM_m^+$, $r>0$ and $\alpha_j,\beta_j\in[0,\infty)$, $j=1,2$, letting $A_j:=A^{\alpha_j}$ and
$B_j:=B^{\beta_j}$ in \eqref{F-2.6} gives
\[
\big\|\,|A^{\theta\alpha_1+(1-\theta)\alpha_2}B^{\theta\beta_1+(1-\theta)\beta_2}|^r\big\|
\le\big\|\,|A^{\alpha_1}B^{\beta_1}|^r\big\|^\theta
\big\|\,|A^{\alpha_2}B^{\beta_2}|^r\big\|^{1-\theta},\qquad0<\theta<1,
\]
that is, the function $(\alpha,\beta)\mapsto\|\,|A^\alpha B^\beta|^r\|$ is jointly log-convex on $[0,\infty)^2$.
This and similar joint log-convexity results in \cite{Hi2,Sa} are all special cases of Bourin and Lee's
extended version \cite[Theorem 1.2]{BL}. Some related results are also found in \cite{BS}.
\end{remark}

%
%

\section{Application to the $\alpha$-$z$-R\'enyi divergence}\label{Sec-3}

In this section we apply the log-majorization given in Sec.~\ref{Sec-2} to the $\alpha$-$z$-R\'enyi divergence
in quantum information. In this section, for $\sigma\in\bM_m^+$ we write $\sigma^0$ for the support projection
of $\sigma$. Moreover, for $t>0$ we consider $\sigma^{-t}$ as $(\sigma^{-1})^t$, where $\sigma^{-1}$ is the
generalized inverse of $\sigma$, i.e., the inverse with restriction to the support of $\sigma$.

Let $\rho,\sigma\in\bM_m^+$, $\alpha\in[0,\infty)$ and $z\in(0,\infty)$. We define
\[
Q_{\alpha,z}(\rho\|\sigma):=\begin{cases}
\Tr\bigl(\rho^{\alpha\over2z}\sigma^{1-\alpha\over z}\rho^{\alpha\over2z}\bigr)^z
=\Tr\,\big|\rho^{\alpha\over2z}\sigma^{1-\alpha\over2z}\big|^{2z}
& \text{if $\alpha\in[0,1]$ or $\rho^0\le\sigma^0$}, \\
+\infty & \text{if $\alpha>1$ and $\rho^0\not\le\sigma^0$}.
\end{cases}
\]
When $\rho\ne0$ and $\alpha\ne1$, the \emph{$\alpha$-$z$-R\'enyi divergence} (also called the
\emph{$\alpha$-$z$-R\'enyi relative entropy}) is defined by
\[
D_{\alpha,z}(\rho\|\sigma):={1\over\alpha-1}\log{Q_{\alpha,z}(\rho\|\sigma)\over\Tr\,\rho}.
\]
Moreover, the \emph{normalized relative entropy} is defined by
\[
D_{1,z}(\rho\|\sigma):=D_1(\rho\|\sigma):={D(\rho\|\sigma)\over\Tr\rho},
\]
where $D(\rho\|\sigma)$ is the \emph{Umegaki relative entropy} (see, e.g., \cite{MH2} for more details on
the $\alpha$-$z$-R\'enyi and the Umegaki relative entropies).

As shown in \cite[Proposition 1]{LT} and \cite[(II.21)]{MH2} it follows from \eqref{F-2.2} applied to the
trace-norm that $D_{\alpha,z}(\rho\|\sigma)$ is monotone in the parameter $z$ in such a way that
\[
0<z_1\le z_2\implies
\begin{cases}D_{\alpha,z_1}(\rho\|\sigma)\le D_{\alpha,z_2}(\rho\|\sigma), & \alpha\in(0,1),\\
D_{\alpha,z_1}(\rho\|\sigma)\ge D_{\alpha,z_2}(\rho\|\sigma), & \alpha>1.
\end{cases}
\]
Furthermore, from \cite[Theorem 2.1]{Hi} (see also Appendix \ref{Sec-A}) it follows that if
$D_{\alpha,z_1}(\rho\|\sigma)=D_{\alpha,z_2}(\rho\|\sigma)<+\infty$ for some $\alpha\in(0,1)\setminus\{1\}$
and some $z_1,z_2>0$ with $z_1\ne z_2$, then we have $\rho\sigma=\sigma\rho$. The next theorem shows
monotonicity of $D_{\alpha,z}(\rho\|\sigma)$ in the parameter $\alpha$.

\begin{theorem}\label{T-3.1}
Let $\rho,\sigma\in\bM_m^+$ with $\rho\ne0$. Then the following hold:
\begin{itemize}
\item[(i)] Assume that $\rho^0\not\perp\sigma^0$, i.e., $\rho^0$ and $\sigma^0$ are not orthogonal. Then for any
$z>0$ the function $\alpha\mapsto Q_{\alpha,z}(\rho\|\sigma)$ is log-convex on $[0,\infty)$, that is, for every
$\alpha_1,\alpha_2\in[0,\infty)$,
\begin{align}\label{F-3.1}
Q_{\theta\alpha_1+(1-\theta)\alpha_2,z}(\rho\|\sigma)
\le Q_{\alpha_1,z}(\rho\|\sigma)^\theta Q_{\alpha_2,z}(\rho\|\sigma)^{1-\theta},\qquad0<\theta<1,
\end{align}
with the convention that $(+\infty)^\theta:=+\infty$.
\item[(ii)] For any $z>0$ the function $\alpha\mapsto D_{\alpha,z}(\rho\|\sigma)$ is monotone increasing on
$[0,\infty)$. In particular, for every $\alpha\in[0,1)$ and $\alpha'\in(1,\infty)$,
\[
D_{\alpha,z}(\rho\|\sigma)\le D_1(\rho\|\sigma)\le D_{\alpha',z}(\rho\|\sigma).
\]
\end{itemize}
\end{theorem}

\begin{proof}
(i)\enspace
By the assumption $\rho^0\not\perp\sigma^0$, note that $0<Q_{\alpha,z}(\rho\|\sigma)\le+\infty$ for all
$\alpha\ge0$, $z>0$. First, assume that $\rho^0\le\sigma^0$. Then for arbitrary $z>0$ and
$\alpha_j\in[0,\infty)$, $j=1,2$, let us apply \eqref{F-2.6} to $A_j:=\rho^{\alpha_j/2z}$,
$B_j:=\sigma^{(1-\alpha_j)/2z}$ ($=(\sigma^{-1})^{(\alpha_j-1)/2z}$ if $\alpha_j>1$), $r=2z$ and the trace-norm.
For any $\theta\in(0,1)$, since (whichever $\alpha_j<1$ or $\alpha_j=1$ or $\alpha_j>1$)
\[
Q_{\theta\alpha_1+(1-\theta)\alpha_2,z}(\rho\|\sigma)
=\Tr\,\Big|\Bigl(\bigl(\rho^{\alpha_1\over2z}\bigr)^\theta
\bigl(\rho^{\alpha_2\over2z}\bigr)^{1-\theta}\Bigr)
\Bigl(\bigl(\sigma^{1-\alpha_1\over2z}\bigr)^\theta
\bigl(\sigma^{1-\alpha_2\over2z}\bigr)^{1-\theta}\Bigr)\Big|^{2z}
\]
and
\[
Q_{\alpha_j,z}(\rho\|\sigma)
=\Tr\,\Big|\rho^{\alpha_j\over2z}\sigma^{1-\alpha_j\over2z}\Big|^{2z},\qquad j=1,2,
\]
we have inequality \eqref{F-3.1} from \eqref{F-2.6}. Next, assume that $\rho^0\not\le\sigma^0$. Then the
above argument remains valid for $\alpha_j$ restricted to $[0,1]$. Since $Q_{\alpha,z}(\rho\|\sigma)=+\infty$
for all $\alpha>1$ and $z>0$ by definition, the RHS of \eqref{F-3.1} is $+\infty$ unless $\alpha_1,\alpha_2<1$.
Hence we have the result.

(ii)\enspace
First, assume that $\rho^0\le\sigma^0$; then $0<Q_{\alpha,z}(\rho\|\sigma)<\infty$ for all $\alpha\ge0$, $z>0$.
Since $\alpha\mapsto Q_{\alpha,z}(\rho\|\sigma)$ is continuous on $[0,\infty)$ and
\begin{align}\label{F-3.2}
\lim_{\alpha\to1}Q_{\alpha,z}(\rho\|\sigma)=\Tr(\rho^{1/2z}\sigma^0\rho^{1/2z})^z=\Tr\,\rho,
\end{align}
it follows from (i) that
\[
D_{\alpha,z}(\rho\|\sigma)={\log Q_{\alpha,z}(\rho\|\sigma)-\log\Tr\,\rho\over\alpha-1}
\]
is increasing in $\alpha\in[0,\infty)\setminus\{1\}$, that is, $D_{\alpha,z}(\rho\|\sigma)\le D_{\alpha',z}(\rho\|\sigma)$
for all $\alpha,\alpha'\in[0,\infty)\setminus\{1\}$ with $\alpha<\alpha'$. Next, assume that $\rho^0\not\le\sigma^0$.
Here we may assume that $\rho^0\not\perp\sigma^0$, since otherwise $D_{\alpha,z}(\rho\|\sigma)=+\infty$ for all
$\alpha\ge0$, $z>0$. Then the above argument can still work whenever $\alpha\in[0,1)$, where \eqref{F-3.2}
is though replaced with $\lim_{\alpha\nearrow1}Q_{\alpha,z}(\rho\|\sigma)<\Tr\,\rho$. Since
$D_{\alpha,z}(\rho\|\sigma)=+\infty$ for all $\alpha\ge1$, $D_{\alpha,z}(\rho\|\sigma)$ is increasing in
$\alpha\in[0,\infty)\setminus\{1\}$ in this case too. Furthermore, recall that
$\lim_{\alpha\to1}D_{\alpha,z}(\rho\|\sigma)=D_1(\rho\|\sigma)$; see \cite[Proposition III.36]{MH2} for a
stronger result. Therefore, the assertion follows.
\end{proof}

\begin{remark}\label{R-3.2}\rm
Special cases of $D_{\alpha,z}(\rho\|\sigma)$ are the \emph{Petz type $\alpha$-R\'enyi divergence}
$D_\alpha(\rho\|\sigma)=D_{\alpha,1}(\rho\|\sigma)$ \cite{Pe} and the
\emph{sandwiched $\alpha$-R\'enyi divergence} $D_\alpha^*(\rho\|\sigma)=D_{\alpha,\alpha}(\rho\|\sigma)$
\cite{MDSFT,WWY}. It is known \cite[Lemma II.2]{MH1} (also \cite[Proposition 5.3(4)]{Hi3}) that
$\alpha\mapsto D_\alpha$ is monotone increasing on $[0,\infty)$, which is included in Theorem \ref{T-3.1}.
Although the same is known for $D_\alpha^*$ in \cite[Theorem 7]{MDSFT} (also \cite[Lemma 8]{BST}), it is
not included in Theorem \ref{T-3.1}. In this connection, it might be a natural problem to extend the theorem in
such a way that $\alpha\mapsto D_{\alpha,z(\alpha)}(\rho\|\sigma)$ is monotone increasing when
$z(\alpha):=\kappa\alpha+z_0$ with $\kappa,z_0\ge0$. Indeed, if $0<\alpha<1$ and $\alpha\le\alpha'$, then
we have
\[
D_{\alpha,z(\alpha)}(\rho\|\sigma)\le D_{\alpha,z(\alpha')}(\rho\|\sigma)
\le D_{\alpha',z(\alpha')}(\rho\|\sigma),
\]
where the first inequality is seen from \cite[Proposition 1]{LT} and \cite[(II.21)]{MH2} and the second is due to
Theorem \ref{T-3.1}. However, the problem seems difficult when $\alpha>1$. For instance, the monotone
increasing of $\alpha\mapsto D_{\alpha,\alpha-1}(\rho\|\sigma)$ on $(1,\infty)$ and that of
$\alpha\mapsto D_{\alpha,1-\alpha}(\rho\|\sigma)$ on $(0,1)$ were shown in \cite[Lemma 25]{RTT} by
computing the derivative, even though the slope of the line $z(\alpha)=1-\alpha$ is negative. Furthermore, note
that the regularized measured R\'enyi divergence for $\alpha\in(0,1/2)$ is equal to
$D_{\alpha,1-\alpha}(\rho\|\sigma)$ (see \cite[Example III.35]{MBV}, \cite[(II.17)]{MH2}), and the monotone
increasing of $D_{\alpha,1-\alpha}(\rho\|\sigma)$ in $\alpha\in(0,1/2)$ is also seen. In this way, the general
situation of the monotonicity of $D_{\alpha,z}$ in two parameters $\alpha,z$ seems quite complicated.
\end{remark}

\section{Multivariate Araki's log-majorization}\label{Sec-4}

By use of Hirschman's strengthening  of Hadamard's three lines theorem, the following inequality (specialized
here to the operator norm) was proved in \cite{SBT} (see also \cite[Appendix]{HKT}): for every
$A_1,\dots,A_n\in\bM_m^+$ and any $\theta\in(0,1]$,
\[
\log\Bigg\|\,\Bigg|\prod_{j=1}^nA_j^\theta\Bigg|^{1/\theta}\Bigg\|_\infty
\le\int_{-\infty}^\infty\log\Bigg\|\prod_{j=1}^nA_j^{1+it}\Bigg\|_\infty\,d\beta_\theta(t),
\]
where
\[
d\beta_\theta(t):={\sin(\pi\theta)\over2\theta(\cos(\pi t)+\cos(\pi\theta))}\,dt,
\]
and the functional calculus $A_j^z$ for any $z\in\bC$ is defined with the convention that $0^z=0$. By extending
the antisymmetric tensor power technique to the setting with a logarithmic integral average, in \cite[(112)]{HKT}
we generalized the above inequality to the form of log-majorization as follows: for every $A_j,\theta$ as above,
\[
\log\lambda\Biggl(\Bigg|\prod_{j=1}^nA_j^\theta\Bigg|^{1/\theta}\Biggr)
\prec\int_{-\infty}^\infty\log\lambda\Bigg(\Bigg|\prod_{j=1}^nA_j^{1+it}\Bigg|\Biggr)\,d\beta_\theta(t),
\]
where we use the convention $\log0:=-\infty$ and majorization $\prec$ makes sense even for vectors having
entries $-\infty$. This may be compared with \eqref{F-2.1} and called the \emph{multivariate Araki's log-majorization}.
In particular, by letting $\theta\searrow0$, this yields, for every Hermitian $H_j\in\bM_m$ and any $r>0$, the
multivariate Golden--Thompson inequality \cite[Corollary 18]{HKT}
\[
\Tr\,\exp\Biggl(r\sum_{j=1}^nH_j\Biggr)
\le\exp\int_{-\infty}^\infty\log\Tr\,\Bigg|\prod_{j=1}^ne^{(1+it)H_j}\Bigg|^r\,d\beta_0(t)
\]
and hence
\begin{align}\label{F-4.1}
\Tr\,\exp\Biggl(r\sum_{j=1}^nH_j\Biggr)
\le\int_{-\infty}^\infty\Tr\,\Bigg|\prod_{j=1}^ne^{(1+it)H_j}\Bigg|^r\,d\beta_0(t).
\end{align}
In fact, when $H_j$'s are replaced with $H_j/2$, inequality \eqref{F-4.1} for $n=2$ and $r=2$ reduces to the
classical Golden--Thompson inequality
\begin{align}\label{F-4.2}
\Tr\,e^{H_1+H_2}\le\Tr\,e^{H_1}e^{H_2}.
\end{align}
For $n=3$ and $r=2$, \eqref{F-4.1} reduces to
\[
\Tr\,e^{H_1+H_2+H_3}\le\int_{-\infty}^\infty\Tr\,e^{H_1}e^{(1+it)H_2/2}e^{H_3}e^{(1-it)H_2/2}\,d\beta_0(t),
\]
which was first shown by Lieb \cite[Theorem 7]{Li} in view of \cite[Lemma 3.4]{SBT}.

From \cite[Theorem 2.1]{Hi} (see also Appendix \ref{Sec-A}) we know that equality holds in the
Golden--Thompson inequality in \eqref{F-4.2} if and only if $H_1H_2=H_2H_1$. Then it might be natural to think
of a similar characterization for the equality case of the multivariate Golden--Thompson inequality in \eqref{F-4.1}.
It is of course obvious that equality holds in \eqref{F-4.1} if $H_jH_k=H_kH_j$ for all $j,k$. However, the converse
direction seems a complicated problem as the next example suggests. In practice, it seems difficult to even
conjecture an exact form of necessary and sufficient condition for the equality case of \eqref{F-4.1}.

\begin{example}\label{E-4.1}\rm
Let us choose non-commuting Hermitian matrices $H,K$, and define
\[
H_1:=H\oplus H,\qquad H_2:=(-H)\oplus(-K),\qquad H_3:=K\oplus K.
\]
Then it is easy to check the following:
\begin{itemize}
\item any pair from $H_1,H_2,H_3$ is non-commuting,
\item any $H_j$ is non-commuting with $H_k+H_l$ for any $k\ne l$,
\item any $H_j$ is non-commuting with $H_1+H_2+H_3$.
\end{itemize}
These mean that there is no non-trivial commutativity relation conceivable for a triplet $(H_1,H_2,H_3)$.
Nevertheless, we compute
\begin{align*}
&\Tr\,e^{H_1+H_2+H_3}=\Tr\,e^{K\oplus H}=\Tr(e^K\oplus e^H)=\Tr\,e^K+\Tr\,e^H, \\
&\Tr\,e^{H_1}e^{H_2}e^{H_3}=\Tr(e^He^{-H}e^K\oplus e^He^{-K}e^K)=\Tr\,e^K+\Tr\,e^H,
\end{align*}
so that $\Tr\,e^{H_1+H_2+H_3}=\Tr\,e^{H_1}e^{H_2}e^{H_3}$ holds. Furthermore, for any $r>0$ we compute
\begin{align*}
&\int_{-\infty}^\infty\Tr\,\big|e^{(1+it)H_1}e^{(1+it)H_2}e^{(1+it)H_3}\big|^r\,d\beta_0(t) \\
&\quad=\int_{-\infty}^\infty\Tr\bigl(e^{H_3}e^{(1-it)H_2}e^{2H_1}
e^{(1+it)H_2}e^{H_3}\bigr)^{r/2}\,d\beta_0(t) \\
&\quad=\int_{-\infty}^\infty\Tr\Bigl(\bigl(e^Ke^{-(1-it)H}e^{2H}e^{-(1+it)H}e^K)
\oplus(e^Ke^{-(1-it)K}e^{2H}e^{-(1+it)K}e^K)\Bigr)^{r/2}\,d\beta_0(t) \\
&\quad=\int_{-\infty}^\infty\Tr\bigl(e^{2K}\oplus e^{itK}e^{2H}e^{-itK}\bigr)^{r/2}\,d\beta_0(t)
=\int_{-\infty}^\infty\Tr\bigl(e^{rK}\oplus e^{itK}e^{rH}e^{-itK}\bigr)\,d\beta_0(t) \\
&\quad=\Tr\bigl(e^{rK}\oplus e^{rH}\bigr)=\Tr\bigl(e^K\oplus e^H\bigr)^r
=\Tr\,e^{r(H_1+H_2+H_3)},
\end{align*}
so that equality holds in \eqref{F-4.1} for all $r>0$.
\end{example}

\section{Log-majorization for the Karcher mean}\label{Sec-5}

In this and the next sections, for convenience, we write $\bH_m$ and $\bP_m$ for the sets of $m\times m$
Hermitian matrices and of $m\times m$ positive definite matrices, respectively. For $\alpha\in[0,1]$ and
$A,B\in\bP_m$ the two-variable \emph{weighted geometric mean} $A\#_\alpha B$ is given as
\[
A\#_\alpha B=A^{1/2}(A^{-1/2}BA^{-1/2})^\alpha A^{1/2},
\]
which is the most studied Kubo--Ando's operator mean \cite{KA}. Let $\omega=(w_1,\dots,w_m)$ be a
weight (probability) vector. The multivariable extension of the geometric mean was developed by several
authors in \cite{BH,Mo} and \cite{LL1,LL2} in the Riemannian geometry approach. For $A_1,\dots,A_n\in\bP_m$
the \emph{Riemannian geometric mean}, called more often the \emph{Karcher mean}, is defined as a unique
minimizer
\[
G_\omega(A_1,\dots,A_n)=\underset{X\in\bP_m}{\mathrm{arg\,min}}
\sum_{j=1}^nw_j\delta^2(X,A_j).
\]
Here, $\delta(A,B)$ denotes the geodesic distance with respect to the Riemannian trace metric, which is
explicitly given as (see \cite{BH,Mo})
\[
\delta(A,B)=\Biggl(\sum_{i=1}^m\log^2\lambda_i(A^{-1}B)\Biggr)^{1/2},\qquad A,B\in\bP_m.
\]
By taking the gradient zero equation for the map $X\in\bP_m\mapsto\delta(X,A)$ as in \cite{Kar}, it was
shown in \cite{BH,Mo} that $G_\omega(A_1,\dots,A_n)$ is a unique solution $X\in\bP_m$ of
\begin{align}\label{F-5.1}
\sum_{j=1}^nw_j\log X^{1/2}A_j^{-1}X^{1/2}=0,
\end{align}
which is called the \emph{Karcher equation} after \cite{Kar}. Note that
$G_{(1-\alpha,\alpha)}(A,B)=A\#_\alpha B$ for $A,B\in\bP_m$. See, e.g., \cite{LL1} for a list of basic properties
of the Karcher mean.

For any $\alpha\in[0,1]$ and $A,B\in\bP_m$ the log-majorization
\begin{align}\label{F-5.2}
A^p\#_\alpha B^p\prec_{\log}(A\#_\alpha B)^p,\qquad p\ge1,
\end{align}
or equivalently,
\[
(A^p\#_\alpha B^p)^{1/p}\prec_{\log}(A^q\#_\alpha B^q)^{1/q},\qquad0<q\le p,
\]
was proved in \cite{AH}. The main part of the proof of \eqref{F-5.2} in \cite{AH} is to prove the operator
norm inequality (sometimes called the \emph{Ando--Hiai inequality})
$\|A^p\#_\alpha B^p\|_\infty\le\|A\#_\alpha B\|_\infty^p$ for all $p\ge1$, or equivalently,
\begin{align}\label{F-5.3}
A\#_\alpha B\le I \implies A^p\#_\alpha B^p\le I\ \ \mbox{for all $p\ge1$}.
\end{align}
Once this is proved, \eqref{F-5.2} is immediately shown by the antisymmetric tensor power technique.

In \cite[Theorem 3]{Ya} Yamazaki proved the multivariate version of \eqref{F-5.3} as follows:
\begin{align}\label{F-5.4}
G_\omega(A_1,\dots,A_n)\le I \implies G_\omega(A_1^p,\dots,A_n^p)\le I\ \ \mbox{for all $p\ge1$}.
\end{align}
Moreover, we note that the Karcher mean $G_\omega$ is well behaved under the antisymmetric tensor
power procedure in such a way that for every $A_1,\dots,A_n\in\bP_m$,
\begin{align}\label{F-5.5}
G_\omega(A_1,\dots,A_n)^{\wedge k}=G_\omega(A_1^{\wedge k},\dots,A_n^{\wedge k}),
\qquad1\le k\le m,
\end{align}
as observed in \cite[Theorem 4.4]{BK} in the uniform weight case, but the proof is valid for any $\omega$
as well. From \eqref{F-5.4} and \eqref{F-5.5} the same argument as in \cite{AH} (also in the proof of
Theorem \ref{T-2.1}) shows the log-majorization for $G_\omega$, as stated in \cite[Remark 3.8]{HP}, which
we state as a theorem here for later use.

\begin{theorem}\label{T-5.1}
For every $A_1,\dots,A_n\in\bP_m$ we have
\begin{align}\label{F-5.6}
G_\omega(A_1^p,\dots,A_n^p)^{1/p}\prec_{\log}G_\omega(A_1^q,\dots,A_n^q)^{1/q},
\qquad0<q\le p,
\end{align}
and
\begin{align}\label{F-5.7}
G_\omega(A_1^p,\dots,A_n^p)^{1/p}\prec_{\log}LE_\omega(A_1,\dots,A_n),\qquad p>0,
\end{align}
where
\[
LE_\omega(A_1,\dots,A_n):=\exp\Biggl(\sum_{j=1}^nw_j\log A_j\Biggr)
\]
is the so-called \emph{Log-Euclidean mean} of $A_1,\dots,A_n$.
\end{theorem}

Log-majorization \eqref{F-5.7} follows from \eqref{F-5.6} and the Lie--Trotter formula
\begin{align}\label{F-5.8}
\lim_{q\to0}G_\omega(A_1^q,\dots,A_n^q)^{1/q}=LE_\omega(A_1,\dots,A_n).
\end{align}
Note that \eqref{F-5.8} is verified as in the proof of \cite[Lemma 3.1]{FFS}. (In \cite{HL} Theorem \ref{T-5.1} and
\eqref{F-5.8} were proved in a more general setting of probability measures.)

Before stating the main theorem of this section, we prepare a technical lemma on power series expansions
related to the Karcher equation. For given $H_1,\dots,H_n\in\bH_m$ and a weight vector
$\omega=(w_1,\dots,w_n)$ with $w_j>0$, we write for each $t\in\bR$,
\begin{align}
X(t)&:=G_\omega(e^{tH_1},\dots,e^{tH_n}), \label{F-5.9}\\
Y(t)&:=X(t)^{1/2}, \nonumber\\
Z_j(t)&:=Y(t)e^{-tH_j}Y(t),\qquad1\le j\le n. \nonumber
\end{align}
Then the Karcher equation \eqref{F-5.1} gives
\begin{align}\label{F-5.10}
\sum_{j=1}^nw_j\log Z_j(t)=\sum_{j=1}^nw_j\log X(t)^{1/2}e^{-tH_j}X(t)^{1/2}=0,\qquad t\in\bR.
\end{align}
It is obvious that $X(0)=Y(0)=Z_j(0)=I$. Here recall that the map
$(A_1,\dots,A_n)\mapsto G_\omega(A_1,\dots,A_n)$ is real analytic in a neighborhood of $(I,\dots,I)$ in
$(\bP_m)^n$, as Lawson stated in \cite[Theorem 7.2]{Law} (in the general $C^*$-algebra setting). Its proof was
indeed given as part of the proof of \cite[Theorem 6.3]{LL2}. From this fact it follows that the matrix function
$X(t)$ is real analytic in a neighborhood of $t=0$, and hence $Y(t)$ and $Z_j(t)$, $1\le i\le m$, are real analytic
near $t=0$ as well. Consequently, we can expand $X(t)$, $Y(t)$ and $Z_j(t)$ in power series, norm convergent
near $t=0$, as follows:
\begin{align}
X(t)&=I+tX_1+t^2X_2+\cdots, \label{F-5.11}\\
Y(t)&=I+tY_1+t^2Y_2+\cdots, \nonumber\\
Z_j(t)&=I+tZ_{1,j}+t^2Z_{2,j}+\cdots,\qquad1\le j\le n, \nonumber
\end{align}
where $X_k$, $Y_k$ and $Z_{k,j}$, $k\in\bN$, are matrices in $\bH_m$. Furthermore, we write
\begin{align}\label{F-5.12}
H^{(k)}:=\sum_{j=1}^nw_jH_j^k,\qquad
Z^{(k)}:=\sum_{j=1}^nw_jZ_{k,j},\qquad k\in\bN.
\end{align}

Below we make some computations based on the above power series of $X(t)$, $Y(t)$ and $Z_j(t)$. All power
series appearing are convergent near $t=0$. At first, since $X(t)=Y(t)^2$, we have
\begin{align}\label{F-5.13}
X_k=\sum_{l=0}^kY_lY_{k-l},\qquad k\in\bN.
\end{align}
Since
\begin{align*}
\log Z_j(t)&=\log\Biggl(I+\sum_{k=1}^\infty t^kZ_{k,j}\Biggr)
=\sum_{r=1}^\infty{(-1)^{r-1}\over r}\Biggl(\sum_{k=1}^\infty t^kZ_{k,j}\Biggr)^r \\
&=\sum_{k=1}^\infty t^k\left(\sum_{r=1}^k{(-1)^{r-1}\over r}\left(
\sum_{k_1,\dots,k_r\ge1\atop{k_1+\cdots+k_r=k}}Z_{k_1,j}\cdots Z_{k_r,j}\right)\right),
\end{align*}
the Karcher equation \eqref{F-5.10} gives
\[
\sum_{r=1}^k{(-1)^{r-1}\over r}\left(\sum_{j=1}^nw_j\left(
\sum_{k_1,\dots,k_r\ge1\atop{k_1+\cdots+k_r=k}}Z_{k_1,j}\cdots Z_{k_r,j}\right)\right)=0,
\qquad k\in\bN,
\]
which implies that
\begin{align}\label{F-5.14}
Z^{(k)}=\begin{cases}\ 0, & k=1, \\
\displaystyle \sum_{r=2}^k{(-1)^r\over r}\left(\sum_{j=1}^nw_j
\left(\sum_{k_1,\dots,k_r\ge1\atop{k_1+\cdots+k_r=k}}Z_{k_1,j}\cdots Z_{k_r,j}\right)\right),
& k\ge2.\end{cases}
\end{align}
Since
\[
Z_j(t)=\Biggl(\sum_{r=0}^\infty t^rY_r\Biggr)\Biggl(\sum_{l=0}^\infty{(-1)^l\over l!}t^lH_j^l\Biggr)
\Biggl(\sum_{s=0}^\infty t^sY_s\Biggr)
\quad(\mbox{where $Y_0=H_j^0=I$}),
\]
we have
\begin{align}
Z_{k,j}&=\sum_{l=0}^k{(-1)^l\over l!}\left(\sum_{r=0}^{k-l}Y_rH_j^lY_{k-l-r}\right) \nonumber\\
&=\begin{cases}2Y_1-H_j, & k=1, \\
\displaystyle 2Y_k+\sum_{r=1}^{k-1}Y_rY_{k-r}+\sum_{l=1}^k{(-1)^l\over l!}
\left(\sum_{r=0}^{k-l}Y_rH_j^lY_{k-l-r}\right), & k\ge2.\end{cases}\label{F-5.15}
\end{align}
This with \eqref{F-5.14} (for $k=1$) gives
\begin{align}\label{F-5.16}
2Y_k=\begin{cases}H^{(1)}, & k=1, \\
\displaystyle Z^{(k)}-\sum_{r=1}^{k-1}Y_rY_{k-r}-\sum_{l=1}^k{(-1)^l\over l!}
\left(\sum_{r=0}^{k-l}Y_rH^{(l)}Y_{k-l-r}\right), & k\ge2.\end{cases}
\end{align}

A crucial point in the above expressions in \eqref{F-5.14}--\eqref{F-5.16} is that they together provide
a recursive way to compute the sequence $\{Y_k\}_{k=1}^\infty$ and hence the sequence
$\{X_k\}_{k=1}^\infty$ by \eqref{F-5.13}. Indeed, $Z^{(1)}$, $Y_1$ and $Z_{1,j}$ are fixed as above,
and if $Z^{(r)}$, $Y_r$ and $Z_{r,j}$ are determined for $1\le r\le k-1$ with $k\ge2$, then
$Z^{(k)}$, $Y_k$ and $Z_{k,j}$ are determined by \eqref{F-5.14}, \eqref{F-5.16} and \eqref{F-5.15} in this
order.

The following lemma summarizes simple but tedious computations of $Y_k$ and $X_k$ up to $k=4$,
which will be used in the proof of the theorem below. The details of computations are left to the reader.

\begin{lemma}\label{L-5.2}
We have
\begin{align*}
2Y_1&=H^{(1)},\qquad\qquad\,X_1=H^{(1)}, \\
2Y_2&={1\over4}\bigl(H^{(1)}\bigr)^2,\qquad\,X_2={1\over2}\bigl(H^{(1)}\bigr)^2, \\
2Y_3&={1\over24}\bigl(H^{(1)}\bigr)^3-{1\over12}H^{(1)}H^{(2)}-{1\over12}H^{(2)}H^{(1)}
+{1\over6}\sum_{j=1}^nw_jH_jH^{(1)}H_j, \\
X_3&={1\over6}\Biggl(\bigl(H^{(1)}\bigr)^3-{1\over2}H^{(1)}H^{(2)}-{1\over2}H^{(2)}H^{(1)}
+\sum_{j=1}^nw_jH_jH^{(1)}H_j\Biggr), \\
2Y_4&={1\over192}(H^{(1)})^4-{1\over48}(H^{(1)})^2H^{(2)}-{1\over24}H^{(1)}H^{(2)}H^{(1)}
-{1\over48}H^{(2)}(H^{(1)})^2 \\
&\qquad\quad+{1\over24}\sum_{j=1}^nw_jH^{(1)}H_jH^{(1)}H_j
+{1\over24}\sum_{j=1}^nw_jH_jH^{(1)}H_jH^{(1)}, \\
X_4&={1\over24}\Biggl((H^{(1)})^4-(H^{(1)})^2H^{(2)}-2H^{(1)}H^{(2)}H^{(1)}
-H^{(2)}(H^{(1)})^2 \\
&\qquad\qquad+2\sum_{j=1}^nw_jH^{(1)}H_jH^{(1)}H_j
+2\sum_{j=1}^nw_jH_jH^{(1)}H_jH^{(1)}\Biggr).
\end{align*}
\end{lemma}

We give one more lemma to use below.

\begin{lemma}\label{L-5.3}
Let $\|\cdot\|$ be a strictly increasing unitarily invariant norm and let $A,B\in\bP_m$. If $A\prec_{\log}B$
and $\|A\|=\|B\|$, then $\lambda(A)=\lambda(B)$, and hence $\Tr\,f(A)=\Tr\,f(B)$ and $\|f(A)\|=\|f(B)\|$ for any
real function $f$ on $(0,\infty)$.
\end{lemma}

\begin{proof}
The first assertion is \cite[Lemma 2.2]{Hi1} and the latter follows since $\lambda(A)=\lambda(B)$ implies that
$\lambda(f(A))=\lambda(f(B))$ and hence $s(f(A))=s(f(B))$.
\end{proof}

The following is our main result, where we characterize the equality case in the norm inequality deriving from
log-majorization \eqref{F-5.7} for the Karcher mean.

\begin{theorem}\label{T-5.4}
Let $A_1,\dots,A_n\in\bP_m$ and $\omega=(w_1,\dots,w_n)$ be a weight vector with $w_j>0$,
$1\le j\le m$. Let $\|\cdot\|$ be a strictly increasing unitarily invariant norm. Then the following
conditions are equivalent:
\begin{itemize}
\item[\rm(a)] $A_j$ commutes with $LE_\omega(A_1,\dots,A_n)$ for any $j=1,\dots,n$;
\item[\rm(b)] $G_\omega(A_1^t,\dots,A_n^t)=LE_\omega(A_1^t,\dots,A_n^t)$ for all $t\in\bR$.
\item[\rm(c)] $\|G_\omega(A_1^t,\dots,A_n^t)\|=\|LE_\omega(A_1^t,\dots,A_n^t)\|$ for some $t\ne0$;
\item[\rm(d)] $\|G_\omega(A_1^t,\dots,A_n^t)^{1/t}\|=\|LE_\omega(A_1,\dots,A_n)\|$  for some $t\ne0$.
\end{itemize}
\end{theorem}

\begin{proof}
Write $A_j=e^{H_j}$ with $H_j\in\bH_m$ for $1\le j\le n$, and use the notations
\begin{align*}
X(t)&:=G_\omega(A_1^t,\dots,A_n^t)=G_\omega\bigl(e^{tH_1},\dots,e^{tH_n}\bigr), \\
H^{(1)}&:=\sum_{j=1}^nw_j\log A_j=\sum_{j=1}^nw_jH_j,
\end{align*}
as in \eqref{F-5.9} and \eqref{F-5.12}, so that $LE_\omega(A_1^t,\dots,A_n^t)=e^{tH^{(1)}}$ for all $t\in\bR$,
and condition (a) is equivalent to that $H_j$ commutes with $H^{(1)}$ for any $j$.

(a)$\implies$(b).\enspace
Assume (a). Then for any fixed $t\in\bR$, letting $M:=e^{tH^{(1)}}$ one has
\[
M^{-1/2}e^{tH_j}M^{-1/2}=e^{t(H_j-H^{(1)})},\qquad1\le j\le n,
\]
so that
\[
\sum_{j=1}^nw_j\log M^{-1/2}e^{tH_j}M^{-1/2}=t\sum_{j=1}^nw_j\bigl(H_j-H^{(1)}\bigr)=0.
\]
Hence (b) holds.

(b)$\implies$(c) and (b)$\implies$(d) are obvious.

(c)$\implies$(d).\enspace
Assume that $\|X(t_0)\|=\big\|e^{t_0H^{(1)}}\big\|$ for some $t_0\ne0$. By replacing $H_j$ with $-H_j$ for all
$j$ if necessary, we may assume $t_0>0$. Since $X(t_0)\prec_{\log}e^{t_0H^{(1)}}$ by \eqref{F-5.7}, it follows
from Lemma \ref{L-5.3} that $\big\|X(t_0)^{1/t_0}\big\|=\big\|e^{H^{(1)}}\big\|$.

(d)$\implies$(a).\enspace
Assume that $\big\|X(t_0)^{1/t_0}\big\|=\big\|e^{H^{(1)}}\big\|$ for some $t_0\ne0$. Since $X(-t_0)^{-1/t_0}=X(t_0)$
by self-duality of $G_\omega$ (see \cite[p.\ 269]{LL1}), we may assume $t_0>0$. Since, by \eqref{F-5.6} and
\eqref{F-5.7},
\[
X(t_0)^{1/t_0}\prec_{\log}X(t)^{1/t}\prec_{\log}e^{H^{(1)}},\qquad t\in(0,t_0],
\]
one has $\|X(t)^{1/t}\|=\big\|e^{H^{(1)}}\big\|$ for all $t\in(0,t_0]$. Hence it follows from Lemma \ref{L-5.3} again
that
\begin{align}\label{F-5.17}
\Tr\,X(t)=\Tr\,e^{tH^{(1)}},\qquad t\in(0,t_0].
\end{align}
Now, we recall the power series expansion of $X(t)$ in \eqref{F-5.11}, which is norm convergent near $t=0$.
By taking the trace we write
\[
\Tr\,X(t)=\Tr\,I+t\Tr\,X_1+t^2\Tr\,X_2+t^3\Tr\,X_3+t^4\Tr\,X_4+\cdots,
\]
as well as
\[
\Tr\,e^{tH^{(1)}}=\Tr\,I+t\Tr\,H^{(1)}+{t^2\over2}\Tr\bigl(H^{(1)}\bigr)^2+{t^3\over6}\Tr\bigl(H^{(1)}\bigr)^3
+{t^4\over24}\Tr(H^{(1)})^4+\cdots.
\]
Then by Lemma \ref{L-5.2} one has
\begin{align*}
\Tr\,X_1&=\Tr\,H^{(1)},\qquad \Tr\,X_2={1\over2}\Tr\bigl(H^{(1)}\bigr)^2, \\
\Tr\,X_3&={1\over6}\Tr\Biggl(\bigl(H^{(1)}\bigr)^3-H^{(1)}H^{(2)}+\sum_{j=1}^nw_jH^{(1)}H_j^2\Biggr)
={1\over6}\Tr\bigl(H^{(1)}\bigr)^3, \\
\Tr\,X_4&={1\over24}\Tr\Biggl(\bigl(H^{(1)}\bigr)^4-4\bigl(H^{(1)}\bigr)^2H^{(2)}
+4\sum_{j=1}^nw_jH^{(1)}H_jH^{(1)}H_j\Biggr),
\end{align*}
where $H^{(2)}:=\sum_{j=1}^nw_jH_j^2$ in the above equations. Therefore, for equality \eqref{F-5.17} to
hold, one must have $\Tr\,X_4={1\over24}\Tr\bigl(H^{(1)}\bigr)^4$ so that
\[
\Tr\bigl(H^{(1)}\bigr)^2H^{(2)}=\Tr\sum_{j=1}w_jH^{(1)}H_jH^{(1)}H_j.
\]
This is equivalently written as
$\Tr\sum_{j=1}^nw_j\bigl(H^{(1)}\bigr)^2H_j^2=\Tr\sum_{j=1}^nw_jH^{(1)}H_jH^{(1)}H_j$, that is,
\[
\sum_{j=1}^nw_j\Tr\bigl(H^{(1)}H_j-H_jH^{(1)}\bigr)^*\bigl(H^{(1)}H_j-H_jH^{(1)}\bigr)=0,
\]
which yields (a).
\end{proof}

\begin{remark}\label{R-5.5}\rm
Theorem \ref{T-5.4} is the multivariate version of \cite[Theorem 2.1]{Hi}. In fact, we note that for the
two-variable case ($n=2$), condition (a) is equivalent to $A_1A_2=A_2A_1$, while for $n>2$, (a) is strictly
weaker than that $A_iA_j=A_jA_i$ for all $i,j$.  However, the condition corresponding to (i) of
\cite[Theorem 3.1]{Hi} is missing in Theorem \ref{T-5.4}, which is condition (e) specified below. In the next
section we discuss this missing condition in connection with the analyticity question of $G_\omega$.
\end{remark}

\section{Analyticity question of the Karcher mean}\label{Sec-6}

In the setting of Theorem \ref{T-5.4}, the multivariate counterpart of condition (i) of \cite[Theorem 3.1]{Hi} is
given as follows:
\begin{itemize}
\item[(e)] $\|G_\omega(A_1^t,\dots,A_n^t)^{1/t}\|$ is not strictly decreasing on $(0,\infty)$.
\end{itemize}
It is obvious that (b) implies (e). Hence any of the conditions of Theorem \ref{T-5.4} implies (e). It is desirable
for us to prove the converse direction. Recall that the real analyticity of $t\mapsto A^t\#_\alpha B^t$ is crucial in
the proof of (i)$\implies$(iv) of \cite[Theorem 3.1]{Hi}. Similarly, we notice that when the real analyticity of
$t\mapsto\Tr\,X(t)$ in $(0,\infty)$, where $X(t):=G_\omega(A_1^t,\dots,A_n^t)$, is known, the proof of
(e)$\implies$(a) is easy as follows: Since $\|X(t)^{1/t}\|$ is decreasing on $(0,\infty)$, condition (e) says that
$\|X(t)^{1/t}\|$ is constant on $[p,q]$ for some $0<p<q$. Since \eqref{F-5.6} gives
$X(q)^{1/q}\prec_{\log}X(t)^{1/t}$ for $0<t\le q$, it follows from Lemma \ref{L-5.3} that  $\Tr\,X(q)^{t/q}=\Tr\,X(t)$
for $t\in[p,q]$. From real analyticity it must hold that $\Tr\,X(q)^{t/q}=\Tr\,X(t)$ for all $t>0$, which is, by
replacing $t$ with $qt$, rewritten as $\Tr\,X(q)^t=\Tr\,X(qt)$ for all $t>0$. Since \eqref{F-5.6} yields
$X(q)^t\prec_{\log}X(qt)$ for $0<t\le1$, one has $\Tr\,X(q)=\Tr\,X(qt)^{1/t}$ for all $t\in(0,1]$ by Lemma \ref{L-5.3}
again. Since $X(qt)^{1/t}\to e^{qH^{(1)}}$ as $t\searrow0$ by the Lie--Trotter formula in \eqref{F-5.8}, we finally
obtain $\Tr\,X(q)=\Tr\,e^{qH^{(1)}}$, which means that condition (c) holds with the trace-norm. Hence (a) follows
since (c)$\implies$(a).

In view of what we have explained above, we are concerned with the following question:
\begin{itemize}
\item[(A)] Is the function $t\mapsto\Tr\,G_\omega(A_1^t,\dots,A_n^t)$ real analytic in $\bR$ for every 
$A_1,\dots,A_n\in\bP_m$?
\end{itemize}
Or, more strongly,
\begin{itemize}
\item[(B)] Is the map $(A_1,\dots,A_n)\mapsto G_\omega(A_1,\dots,A_n)$ real analytic in the whole
$(\bP_m)^n$?
\end{itemize}

As for question (B), we have known so far that $G_\omega(A_1,\dots,A_n)$ is separately real analytic in
each variable and locally real analytic in a neighborhood of each diagonal $(A,\dots,A)$; see \cite{Law,LL2}.
Besides these, a detailed analysis was made in \cite{LL3} in the restricted setting of the three-variable Karcher
mean $G(A,B,C)$ (with the uniform weight) of $2\times2$ $A,B,C$. Question (B) seems indispensable from
the point of view of non-commutative function theory (see, e.g., \cite{KV} and \cite[Part Two]{AMY}) as far as
we think of $G_\omega$ as a good example of non-commutative functions.

Below let us give a brief explanation on difficulty in question (B) when we follow the method in the proof of
\cite[Theorem 6.3]{LL2}. Define the map $F:(\bP_m)^n\times\bP_m\to\bH_m$ by
\begin{align}\label{F-6.1}
F(A_1,\dots,A_m,X)=F_{A_1,\dots,A_n}(X):=\sum_{j=1}^nw_j\log X^{1/2}A_j^{-1}X^{1/2}.
\end{align}
Then $X=G_\omega(A_1,\dots,A_n)$ if and only if $F(A_1,\dots,A_n,X)=0$, and it is clear that $F$ is real
analytic in $(\bP_m)^n\times\bP_m$. Thus, by the real analytic implicit function theorem (see, e.g.,
\cite[Theorem 2.3.5]{KP}), to settle (B), it suffices to prove that the Fr\'echet derivative
$DF_{A_1,\dots,A_n}(X)$ at any $X\in\bP_m$ for each $(A_1,\dots,A_n)\in(\bP_m)^n$ is invertible. For
arbitrary $(A_1,\dots,A_n)$ define
\begin{align}\label{F-6.2}
F_j(X):=\log X^{1/2}A_j^{-1}X^{1/2}=\eta\circ\psi_j\circ\phi(X),
\end{align}
where $\phi(X):=X^{1/2}$, $\psi_j(Y):=YA_j^{-1}Y$ and $\eta(Z):=\log Z$. Then we can write the Fr\'echet
derivative of $F_j$ at $X$ as
\[
DF_j(X)=D\eta(X^{1/2}A_j^{-1}X^{1/2})\circ D\psi_j(X^{1/2})\circ D\phi(X),
\]
so that $DF_{A_1,\dots,A_n}(X)=\sum_{j=1}^nw_jDF_j(X)$. By Daleckii and Krein's derivative formula (see
\cite[Theorem V.3.3]{Bh}) one can see that $D\phi(X)$ and $D\eta(X^{1/2}A_j^{-1}X^{1/2})$ are invertible.
From the classical theory of Lyapunov equation one can also see that $D\psi_j(X^{1/2})$ is invertible. Hence
it follows that each $DF_j(X)$ is invertible. However, it is unknown to us how to show that
$DF_{A_1,\dots,A_n}(X)$ is invertible, while it is clear from the above argument that $G_\omega$ enjoys the
separate real analyticity and the local real analyticity on the diagonals.

\begin{remark}\label{R-6.1}\rm
Let $A_1,\dots,A_n\in\bP_m$. For any weight vector $\omega$ and $t\in[-1,1]\setminus\{0\}$, the
\emph{power mean} $P_{t,\omega}(A_1,\dots,A_n)$ was introduced in \cite{LP} as a unique solution
$X\in\bP_m$ of
\[
\sum_{j=1}^nw_j(X\#_tA_j)=X,\quad\mbox{i.e.,}\quad\sum_{j=1}^nw_j(X^{-1/2}A_jX^{-1/2})^t=I
\]
for $t\in(0,1]$, and $P_{t,\omega}(A_1,\dots,A_n):=\bigl(P_{-t,\omega}(A_1^{-1},\dots,A_n^{-1})\bigr)^{-1}$
for $t\in[-1,0)$. Note \cite{LP} that $P_{t,\omega}\to G_\omega$ as $t\to0$. Replacing \eqref{F-6.1} and
\eqref{F-6.2} with
\[
F(A_1,\dots,A_m,X):=\sum_{j=1}^nw_j\bigl((X^{-1/2}A_jX^{-1/2})^t-I\bigr)
\]
and $F_j(X):=\eta\circ\psi_j\circ\phi(X)$ where $\phi(X):=X^{-1/2}$, $\psi_j(Y):=YA_jY$ and $\eta(Z):=Z^t$,
one can argue in the same way as above. Then one can see that $P_{t,\omega}$ is separately real analytic
in each variable and locally real analytic on the diagonals, while question (B) for $P_{t,\omega}$ is unresolved
similarly to the question for $G_\omega$.
\end{remark}

\subsection*{Acknowledgments}

The author thanks Jean-Christophe Bourin and Mil\'an Mosonyi for discussions which helped to improve
the paper.

\appendix

\section{Correction of the proof of \cite[Theorem 2.1]{Hi}}\label{Sec-A}

In \cite[Theorem 2.1]{Hi} we characterized the equality case of the matrix norm inequality in \eqref{F-2.2}
when a unitarily invariant norm $\|\cdot\|$ is strictly increasing. The main part of the proof of
\cite[Theorem 2.1]{Hi} is to prove that if $\big\|(A^{p/2}B^pA^{p/2})^{1/p}\big\|$ is not strictly increasing in
$p>0$, then $AB=BA$. However, the sentence, four lines below (2.2) in \cite{Hi}, saying
\begin{align}\label{F-A.1}
\mbox{``So taking $RA_0R$ and $RB_0R$ instead of $A_0$ and $B_0$, we can assume that $A_0,B_0>0$''}
\end{align}
is not at all easy to see. The aim of this appendix is to make the proof complete.

To do so, we need the following well-known Lie--Trotter--Kato formula with a continuous parameter, originally
due to \cite{Ka} in the general unbounded operator setting, and related arguments are also found in \cite{HP0}
and \cite{Hi0}.

\begin{lemma}\label{L-A.1}
For every $A,B\in\bM_m^+$,
\[
\lim_{t\searrow0}\bigl(A^{t/2}B^tA^{t/2}\bigr)^{1/t}
=P_0\exp(P_0(\log A)P_0+P_0(\log B)P_0),
\]
where $P_0:=A^0\wedge B^0$, i.e., the orthogonal projection onto the intersection of the ranges of $A,B$.
\end{lemma}

\noindent
{\bf Proof of \cite[Theorem 2.1]{Hi}.}\quad
Below let us correct the proof of (i)$\implies$(v) of \cite[Theorem 2.1]{Hi}. To do this, let us suppose that
$\big\|(A^{p/2}B^pA^{p/2})^{1/p}\big\|=\big\|(A^{q/2}B^qA^{q/2})^{1/q}\big\|$ for some $0<p<q$, and start in
the same way as in the proof up to (2.2) in \cite{Hi}. Next, set $A_0:=A^p$ and $B_0:=B^p$, and let
$P$, $Q$ and $R$ denote the support projections $A_0^0$, $B_0^0$ and
$\bigl(A_0^{1/2}B_0A_0^{1/2}\bigr)^0$, respectively. Since $\bigl(A_0^{1/2}B_0A_0^{1/2}\bigr)^t\to R$ and
$A_0^{t/2}B_0^tA_0^{t/2}\to PQP$ as $t\searrow0$, we have $\lambda(R)=\lambda(PQP)$ by (2.2) of \cite{Hi}.
This implies that $PQP$ is an orthogonal projection, so that $PQ(I-P)QP=PQP-(PQP)^2=0$. Hence
$(I-P)QP=0$ so that
\begin{align}\label{F-A.2}
PQ=QP=P_0,
\end{align}
where $P_0:=P\wedge Q$. Noting that
$\ker\bigl(A_0^{1/2}B_0A_0^{1/2}\bigr)=\ker A_0+(\mathrm{ran}\,A_0\cap\ker B_0)$, we have
\[
R^\perp=P^\perp+Q^\perp P=I-P+(I-Q)P=I-P_0
\]
and hence
\begin{align}\label{F-A.3}
R=P_0=P\wedge Q.
\end{align}

From now on we modify the proof of \cite[Theorem 2.1]{Hi}. Let $H:=P\log A_0$ and $K:=Q\log B_0$. Then for every
$t>0$, we have $tH=P\log A_0^t$ so that
\begin{align}\label{F-A.4}
A_0^t=P+tH+{t^2\over2}H^2+{t^3\over3!}H^3+{t^4\over4!}H^4+\cdots,
\end{align}
and similarly,
\begin{align}\label{F-A.5}
B_0^t=Q+tK+{t^2\over2}K^2+{t^3\over3!}K^3+{t^4\over4!}K^4+\cdots.
\end{align}
Note by (2.2) of \cite{Hi} that $\lambda\bigl(\bigl(A_0^{t/2}B_0^tA_0^{t/2}\bigr)^{1/t}\bigr)$ is independent of any
$t>0$. Since Lemma \ref{L-A.1} gives
\[
\lim_{t\searrow0}\lambda\bigl(\bigl(A_0^{t/2}B_0^tA_0^{t/2}\bigr)^{1/t}\bigr)
=\lambda\bigl(P_0e^{P_0HP_0+P_0KP_0}\bigr),
\]
we have
\[
\lambda\bigl(\bigl(A_0^{t/2}B_0^tA_0^{t/2}\bigr)^{1/t}\bigr)
=\lambda\bigl(P_0e^{P_0HP_0+P_0KP_0}\bigr),\qquad t>0,
\]
and hence
\[
\lambda\bigl(A_0^{t/2}B_0^tA_0^{t/2}\bigr)
=\lambda\bigl(P_0e^{t(P_0HP_0+P_0KP_0)}\bigr),\qquad t>0.
\]
In particular,
\begin{align}\label{F-A.6}
\Tr\,A_0^tB_0^t=\Tr\,P_0e^{t(P_0HP_0+P_0KP_0)},\qquad t>0.
\end{align}
By \eqref{F-A.4} and \eqref{F-A.5} the LHS of \eqref{F-A.6} is
\begin{align}
&\Tr\biggl[PQ+t(QH+PK)+{t^2\over2}(QH^2+2HK+PK^2) \nonumber\\
&\qquad+{t^3\over6}(QH^3+3H^2K+3HK^2+PK^3) \nonumber\\
&\qquad+{t^4\over24}(QH^4+4H^3K+6H^2K^2+4HK^3+PK^4)+o\bigl(t^4\bigr)\biggr]. \label{F-A.7}
\end{align}
On the other hand, the RHS of \eqref{F-A.6} is
\begin{align}
&\Tr\biggl[P_0+t(P_0HP_0+P_0KP_0)+{t^2\over2}(P_0HP_0+P_0KP_0)^2
+{t^3\over6}(P_0HP_0+P_0KP_0)^3 \nonumber\\
&\qquad+{t^4\over24}(P_0HP_0+P_0KP_0)^4+o\bigl(t^4\bigr)\biggr] \nonumber\\
&=\Tr\biggl[P_0+t(P_0H+P_0K)+{t^2\over2}\bigl((P_0H)^2+2P_0HP_0K+(P_0K)^2\bigr) \nonumber\\
&\qquad\quad+{t^3\over6}\bigl((P_0H)^3+3(P_0H)^2P_0K+3P_0H(P_0K)^2+(P_0K)^3\bigr) \nonumber\\
&\qquad\quad+{t^4\over24}\bigl((P_0H)^4+4(P_0H)^3P_0K+4(P_0H)^2(P_0K)^2+2P_0HP_0KP_0HP_0K
\nonumber\\
&\qquad\qquad\qquad+4P_0H(P_0K)^3+(P_0K)^4\bigr)+o\bigl(t^4\bigr)\biggr] \nonumber\\
&=\Tr\biggl[PQ+t(QH+PH)+{t^2\over2}\bigl((QH)^2+2HK+(PK)^2\bigr) \nonumber\\
&\qquad\quad+{t^3\over6}\bigl((QH)^3+3(QH)^2PK+3QH(PK)^2+(PK)^3\bigr) \nonumber\\
&\qquad\quad+{t^4\over24}\bigl((QH)^4+4(QH)^3PK+4(QH)^2(PK)^2+2QHPKQHPK \nonumber\\
&\qquad\qquad\qquad+4QH(PK)^3+(PK)^4\bigr)+o\bigl(t^4\bigr)\biggr], \label{F-A.8}
\end{align}
where the last equality follows from \eqref{F-A.2} and $PH=H$, $QK=K$. Since \eqref{F-A.7} and \eqref{F-A.8}
are equal for all $t>0$, comparing first the coefficients of $t^2$ gives
\[
\Tr(QH^2+PK^2)=\Tr\bigl((QH)^2+(PK)^2\bigr).
\]
Since $\Tr\,QH(I-Q)HQ\ge0$, we have $\Tr\,QH^2\ge\Tr(QH)^2$, and similarly $\Tr\,PK^2\ge\Tr(PK)^2$.
Therefore, the above equality implies that $\Tr\,QH^2=\Tr(QH)^2$ and $\Tr\,PK^2=\Tr(PK)^2$. Then, since
$\Tr\,QH(I-Q)HQ=0$, we have $(I-Q)HQ=0$ so that
\begin{align}\label{F-A.9}
HQ=QH,
\end{align}
and similarly,
\begin{align}\label{F-A.10}
KP=PK.
\end{align}
By \eqref{F-A.2}, \eqref{F-A.9} and \eqref{F-A.10}, the term of $t^4$ in \eqref{F-A.8} becomes
\[
{t^4\over24}\Tr\bigl(QH^4+4H^3K+4H^2K^2+2HKHK+4HK^3+PK^4\bigr).
\]
Compare this with the term of $t^4$ in \eqref{F-A.7} to obtain
\[
\Tr\,H^2K^2=\Tr\,HKHK,
\]
which implies that
\[
\Tr(HK-KH)^*(HK-KH)=2\Tr(H^2K^2-HKHK)=0.
\]
Therefore, $HK=KH$. From this with \eqref{F-A.2}, \eqref{F-A.4} and \eqref{F-A.5} we have $A_0B_0=B_0A_0$,
equivalently, $AB=BA$.\qed

\begin{remark}\label{R-A.2}\rm
From \eqref{F-A.3}, \eqref{F-A.9} and \eqref{F-A.10} we eventually see that the statement in \eqref{F-A.1}
holds true in itself. Indeed, since $HP_0^\perp=H(P^\perp Q+Q^\perp)=HQ^\perp$ and
$KP_0^\perp=K(Q^\perp P+P^\perp)=KP^\perp$, we have $(HP_0^\perp)(KP_0^\perp)=HQ^\perp KP^\perp=0$
and $(KP_0^\perp)(HP_0^\perp)=KP^\perp HQ^\perp=0$. Therefore, in particular, $HP_0^\perp$ and
$KP_0^\perp$ commute, and the question reduces to that $HP_0$ and $KP_0$ commute, where $P_0=R$ by
\eqref{F-A.3}.
\end{remark}

\begin{remark}\label{R-A.3}\rm
The (weak) log-majorization in \eqref{F-2.1} and the norm inequality in \eqref{F-2.2} were extended in \cite{Hi0}
to infinite-dimensional Hilbert space operators. Furthermore, the Araki--Lieb--Thirring trace inequality was
extended by Kosaki \cite{Ko} to the von Neumann algebra setting. It is interesting to pursue the equality case
of those trace/norm inequalities in such more general settings.
\end{remark}

\subsection*{Declarations}

{\bf Conflict of interest}\enspace
The author has no conflicts to disclose.


\begin{thebibliography}{99}

\bibitem{AMY}
J. Agler, J. E. McCarthy and N. Young, {\it Operator Analysis: Hilbert Space Methods in Complex Analysis},
Cambridge University Press, Cambridge, 2020.

\bibitem{An} T. Ando, Majorization, doubly stochastic matrices, and
comparison of eigenvalues, {\it Linear Algebra Appl.} {\bf 118} (1989), 163--248.

\bibitem{AH}
T. Ando and F. Hiai, Log majorization and complementary Golden-Thompson type inequalities,
{\it Linear Algebra Appl.} {\bf 197/198} (1994), 113--131.


\bibitem{Ar}
H. Araki, On an inequality of Lieb and Thirring, {\it Lett. Math. Phys.} {\bf 19}
(1990), 167--170.

%

\bibitem{BST}
M. Berta, V.B. Scholz and M. Tomamichel, R\'enyi divergences as weighted non-commutative
vector valued $L_p$-spaces, {\it Ann. Henri Poincar\'e} {\bf 19} (2018), 1843--1867.

\bibitem{Bh}
R. Bhatia, {\it Matrix Analysis}, Springer-Verlag, New York, 1996.

\bibitem{BH}
R. Bhatia and J. Holbrook, Riemannian geometry and matrix geometric means,
{\it Linear Algebra Appl.} {\bf 413} (2006), 594--618.


\bibitem{BK}
R. Bhatia and R. L. Karandikar, Monotonicity of the matrix geometric mean,
{\it Math. Ann.} {\bf 353} (2012), 1453--1467.



\bibitem{BL}
J.-C. Bourin and E.-Y. Lee, Matrix inequalities from a two variables functional,
{\it Internat. J. Math.} {\bf 27} (2016), 1650071, 19~pp.

\bibitem{BS}
J.-C. Bourin and J. Shao, Convex maps on $\bR^n$ and positive definite matrices,
{\it Comptes Rendus Math.} {\bf 358} (2020), 645--649.


\bibitem{FFS}
J. I. Fujii, M. Fujii and Y. Seo, The Golden--Thompson--Segal type inequalities related to the
weighted geometric mean due to Lawson--Lim, {\it J. Math. Inequal.} {\bf 3} (2009), 511--518.

\bibitem{Hi}
F. Hiai, Equality cases in matrix norm inequalities of Golden-Thompson type,
{\it Linear and Multilinear Algebra} {\bf 36} (1994), 239--249.

\bibitem{Hi0}
F. Hiai, Log-majorizations and norm inequalities for exponential operators, in
{\it Linear Operators}, J. Janas, F. H. Szafraniec and J. Zem\'anek (eds.), Banach Center
Publications, Vol. 38, 1997, pp. 119--181.

\bibitem{Hi1}
F. Hiai, Matrix Analysis: Matrix monotone functions, matrix means, and majorization,
{\it Interdisciplinary Information Sciences} {\bf 16} (2010), 139--248.

\bibitem{Hi2}
F. Hiai, A generalization of Araki's log-majorization, {\it Linear Algebra Appl.} {\bf 501} (2016), 1--16.

\bibitem{Hi3}
F. Hiai, Quantum $f$-divergences in von Neumann algebras I. Standard $f$-divergences,
{\it J. Math. Phys.} {\bf 59} (2018), 102202, 27~pp.

\bibitem{HKT}
F. Hiai, R. K\"onig and M. Tomamichel, Generalized log-majorization and multivariate trace inequalities,
{\it Ann. Henri Poincar\'e} {\bf 18} (2017), 2499--2521.

\bibitem{HL}
F. Hiai and Y. Lim, Log-majorization and Lie--Trotter formula for the Cartan barycenter on probability
measure spaces, {\it J. Math. Anal. Appl.} {\bf 453} (2017), 195--211.

%

\bibitem{HP0}
F. Hiai and D. Petz,
The Golden-Thompson trace inequality is complemented,
{\it Linear Algebra Appl.} {\bf 181} (1993), 153--185.

\bibitem{HP}
F. Hiai and D. Petz, Riemannian metrics on positive definite matrices related to means. II,
{\it Linear Algebra Appl.} {\bf 436} (2012), 2117--2136.

\bibitem{KV}
D. S. Kaliuzhnyi-Verbovetskyi and V. Vinnikov, {\it Foundations of Free Noncommutative
Function Theory}, Mathematical Surveys and Monographs, American Mathematical Society,
Providence, RI, 2014.

\bibitem{Kar}
H. Karcher, Riemannian center of mass and mollifier smoothing,
{\it Comm. Pure Appl. Math.} {\bf 30} (1977), 509--541.

\bibitem{Ka}
T. Kato, Trotter's product formula for an arbitrary pair of self-adjoint contraction semigroups, in
{\it Topics in Functional Analysis}, I. Gohberg and M. Kac (eds.), Academic Press, New York, 1978,
pp. 185--195.


\bibitem{Ko}
H. Kosaki, An inequality of Araki--Lieb--Thirring (von Neumann algebra case),
{\it Proc. Amer. Math. Soc.} {\bf 114} (1992) 477--481.

\bibitem{KP}
S. G. Krantz and H. R. Parks, {\it A Primer of Real Analytic Functions}, 2nd edition,
Birkh\"auser, Boston, 2002.

\bibitem{KA}
F. Kubo and T. Ando, Means of positive linear operators, {\it Math. Ann.} {\bf 246} (1980), 205--224.

\bibitem{Law}
J. Lawson, Existence and uniqueness of the Karcher mean on unital $C^*$-algebras,
{\it J. Math. Anal. Appl.} {\bf 483} (2020), 123625, pp. 16.

\bibitem{LL1}
J. Lawson and Y. Lim, Monotonic properties of the least squares mean,
{\it Math. Ann.} {\bf 351} (2011), 267--279.

\bibitem{LL2}
J. Lawson and Y. Lim, Karcher means and Karcher equations of positive definite operators,
{\it Trans. Amer. Math. Soc. Series B} {\bf 1} (2014), 1--22.

\bibitem{LL3}
J. Lawson and Y. Lim, Analyticity of the Karcher mean coefficient maps,
{\it Linear Algebra Appl.} {\bf 627} (2021), 162--184.

\bibitem{Li}
E. H. Lieb, Convex trace functions and the Wigner--Yanase--Dyson conjecture,
{\it Adv. Math.} {\bf 11} (1973), 267--288.

\bibitem{LP}
Y. Lim and M. P\'alfia, Matrix power means and the Karcher mean,
{\it J. Funct. Anal.} {\bf 262} (2012), 1498--1514.

\bibitem{LT}
S. M. Lin and M. Tomamichel, Investigating properties of a family of quantum R\'enyi divergences,
{\it Quantum Inf. Process.} {\bf 14} (2015), 1501--1512.

\bibitem{MOA}
A. W. Marshall, I. Olkin and B. C. Arnold,
{\it Inequalities: Theory of Majorization and Its Applications}, Second ed.,
Springer-Verlag, New York, 2011.


\bibitem{Mo}
M. Moakher, A differential geometric approach to the geometric mean of symmetric positive-definite matrices,
{\it SIAM J. Matrix Anal. Appl.} {\bf 26} (2005), 735--747.

\bibitem{MBV}
M. Mosonyi, G. Bunth and P. Vrana, Geometric relative entropies and barycentric R\'enyi divergences,
arXiv:2207.14282v5 [quant-ph], 2024.

\bibitem{MH1}
M. Mosonyi and F. Hiai, On the quantum R\'enyi relative entropies and related capacity formulas,
{\it IEEE Trans. Inform. Theory} {\bf 57} (2011), 2474--2487.

\bibitem{MH2}
M. Mosonyi and F. Hiai, Some continuity properties of quantum R\'enyi divergences,
{\it IEEE Trans. Inform. Theory} {\bf 70} (2024), 2674--2700.

\bibitem{MDSFT}
M. M\"uller-Lennert, F. Dupuis, O. Szehr, S. Fehr and M. Tomamichel,
On quantum R\'enyi entropies: A new generalization and some properties,
{\it J. Math. Phys.} {\bf 54} (2013), 122203.

\bibitem{Pe}
D. Petz, Quasi-entropies for finite quantum systems,
{\it Rep. Math. Phys.} {\bf 23} (1986), 57--65.

\bibitem{PW}
W. Pusz and S. L. Woronowicz, Functional calculus for sesquilinear forms and the purification map,
{\it Rep. Math. Phys.} {\bf 8} (1975), 159--170.

\bibitem{RTT}
R. Rubboli, R. Takagi and M. Tomamichel, Mixed-state additivity properties of magic monotones
based on quantum relative entropies for single-qubit states and beyond,
arXiv:2307.08258v2 [quant-ph], 2023.

\bibitem{Sa}
M. Sababheh, Interpolated inequalities for unitarily invariant norms,
{\it Linear Algebra Appl.} {\bf 475} (2015), 240--250.

\bibitem{SBT}
D. Sutter, M. Berta and M. Tomamichel, Multivariate trace inequalities,
{\it Comm. Math. Phys.} {\bf 352} (2017), 37--58.


\bibitem{WWY}
M. M. Wilde, A. Winter and D. Yang,
Strong converse for the classical capacity of entanglement-breaking
and Hadamard channels via a sandwiched R\'enyi relative Entropy,
{\it Comm. Math. Phys.} {\bf 331} (2014), 593--622.

\bibitem{Ya}
T. Yamazaki, The Riemannian mean and matrix inequalities related to the Ando--Hiai inequality
and chaotic order, {\it Oper. Matrices} {\bf 6} (2012), 577--588.

\end{thebibliography}
\end{document}